\documentclass[12 pt]{amsart}

\usepackage[english]{babel}
\usepackage{amssymb,amsmath,amsfonts,amscd}
\usepackage{bbm}
\usepackage{mathptmx}

\usepackage[english]{babel}
\usepackage{amssymb,amsmath, amsthm, amscd, latexsym, epsfig}
\usepackage[latin1]{inputenc}
\usepackage[T1]{fontenc}

\newcommand{\feit}[1]{\mathbf{#1}}
\newcommand{\mon}[1]{\feit{x}^{\feit{#1}}}

\newcommand{\polr}[3]{\field{#1}[#2_1,\ldots,#2_{#3}]}
\newcommand{\field}[1]{\mathbbm{#1}}
\newcommand{\hide}[1]{}
\newcommand{\Bb}[2]{\beta_{#1,#2}}
\newcommand{\BbModulk}[3]{\beta_{#1,#2}(#3;\field{k})}
\newcommand{\Bnumk}[3]{\beta_{#1,#2}(I_{#3};\field{k})}

\newcommand{\Bnum}[3]{\beta_{#1,#2}(I_{#3})}
\newcommand{\nnn}{\mathbb{N}_{0}^{n}}
\newcommand{\nn}{\mathbb{N}_{0}}

\newcommand{\dd}{\Delta}
\newcommand{\ind}[1]{I(#1)}
\newcommand{\bas}[1]{B(#1)}

\newcommand{\ino}[2]{\mathcal{F}_{#2}(#1)}
\newcommand{\ff}[2]{f_{#2}(#1)}
\newcommand{\dual}[1]{\overline{#1}}
\newcommand{\tru}[2]{#1^{(#2)}}
\newcommand{\trures}[3]{{#1^{(#2)}}_{|#3}}
\newcommand{\elo}[2]{#1_{(#2)}}
\newcommand{\iskel}[2]{#1^{(#2)}}
\newcommand{\hilb}[1]{H(#1)}
\newcommand{\srid}[1]{I_{#1}}
\newcommand{\srring}[1]{S/\srid{#1}}
\newcommand{\Bringk}[3]{\beta_{#1,#2}(\srring{#3};\field{k})}
\newcommand{\Bring}[3]{\beta_{#1,#2}(\srring{#3})}
\newcommand{\BringKropp}[4]{\beta_{#1,#2}(\srring{#3};#4)}

\DeclareMathOperator{\ppd}{p.d.\,}
\DeclareMathOperator{\depth}{depth\,}
\newcommand{\pd}[1]{\ppd\,\srid{#1}}
\DeclareMathOperator{\Kdim}{dim\,}

\DeclareMathOperator{\im}{im} 
\DeclareMathOperator{\sign}{sign}

\theoremstyle{plain}
\newtheorem{theorem}{Theorem}[section]
\newtheorem{lemma}{Lemma}[section]
\newtheorem{proposition}{Proposition}[section]
\newtheorem{corollary}{Corollary}[section]

\theoremstyle{definition}
\newtheorem{definition}{Definition}[section]
\newtheorem{example}{Example}[section]

\theoremstyle{remark}
\newtheorem*{remark}{Remark}

\begin{document}
\author{Jan Roksvold}
\title{Betti numbers of skeletons}
\address{Department of Education, UiT The Arctic University of Norway \\ N-9037 Troms\o, Norway}
\email{jan.n.roksvold@uit.no (Corresponding author)}
\author{Hugues Verdure}
\address{Department of Mathematics, UiT The Arctic University of Norway \\ N-9037 Troms\o, Norway}
\email{hugues.verdure@uit.no}


\begin{abstract}
\noindent We demonstrate that the Betti numbers associated to an $\nn$-graded minimal free resolution of the Stanley-Reisner ring $\srring{\iskel{\dd}{d-1}}$ of the $(d-1)$-skeleton of a simplicial complex $\Delta$ of dimension $d$ can be expressed as a $\mathbb{Z}$-linear combination of the corresponding Betti numbers of $\Delta$. An immediate implication of our main result is that the projective dimension of $\srring{\iskel{\dd}{d-1}}$ is at most one greater than the projective dimension of $\srring{\dd}$, and it thus provides a new and direct proof of this. Our result extends immediately to matroids and their truncations. A similar result for matroid elongations can not be hoped for, but we do obtain a weaker result for these. The result does not apply to generalized skeleton ideals.  
\end{abstract}
\maketitle
\section{Introduction}
\noindent In this paper we investigate certain aspects of the relationship between an $\nn$-graded minimal free resolution of the Stanley-Reisner ring of a simplicial complex and those associated to its skeletons. Our main result is Theorem \ref{mainThe}, which says that each of the Betti numbers associated to an $\nn$-graded minimal free resolution of $\srring{\iskel{\dd}{d-1}}$, where $\srid{\iskel{\dd}{d-1}}$ is the ideal generated by monomials corresponding to nonfaces of the $(d-1)$-skeleton of a finite simplicial complex $\dd$, can be expressed as a $\mathbb{Z}$-linear sum of the Betti numbers associated to $\srring{\dd}$. 

Previous results on the Stanley-Reisner rings of skeletons include the classic \cite[Corollary 2.6]{Hib} which states that \begin{equation}\label{Hibis}\depth\srring{\dd}=\max\{j:\iskel{\dd}{j-1} \text{ is Cohen-Macaulay}\}.\end{equation} This result was later generalized to arbitrary monomial ideals in \cite[Corollary 2.5]{HJZ} (we shall return to this in our final section, where we give a counterexample showing that our main result can not be generalized in the same way). By the Auslander-Buchsbaum identity, it follows from (\ref{Hibis}) that \[\pd{\dd}\leq\ppd{\srring{\iskel{\dd}{d-1}}}\leq 1+ \ppd{\srring{\dd}}.\] From the latter of these inequalities it is easily demonstrated, again by using the Auslander-Buchsbaum identity, that every skeleton of a Cohen-Macaulay simplicial complex is Cohen-Macaulay - a fact which was proved in \cite[Corollary 2.5]{Hib} as well.  

That $\ppd{\srring{\iskel{\dd}{d-1}}}\leq 1+\ppd{\srring{\dd}}$ can also be seen as an immediate consequence of our main result, and Theorem \ref{main} thus provides a new and direct proof of this and therefore also of the fact that the Cohen-Macaulay property is inherited by skeletons.

The projective dimension of Stanley-Reisner rings has seen recent research interest. Most notably, it was demonstrated in \cite[Corollary 3.33]{MV} that \[\ppd{\srring{\dd}}\geq \max\{|C|:C \text{ is a circuit of the Alexander dual } \dd^*\text{ of }\dd\},\] with equality if $\srring{\dd}$ is sequentially Cohen-Macaulay.

Our main result extends immediately to a matroid $M$ and its truncations. Such matroid truncations have themselves seen recent research interest. Examples of this are \cite{JP}, which contains the strengthening of a result by Brylawski \cite[Proposition 7.4.10]{Bry} concerning the representability of truncations, and \cite[Proposition 15]{Bri}, where it is demonstrated that the Tutte polynomial of $M$ determines that of its truncation $\tru{M}{1}$.  

Corresponding to our main result applied to matroid truncations, we give a considerably weaker result concerning matroid elongations. It says that the Betti table associated to the elongation of $M$ to rank $r(M)+1$ is equal to the Betti table obtained by removing the second column from the Betti table of $\srring{M}$ - but only in terms of zeros and nonzeros.  

\subsection{Structure of this paper}
\begin{itemize} 
\item In Section \ref{Secprelim} we provide definitions and results used later on. 
\item In Section \ref{SecSimp} we demonstrate that the Betti numbers associated to a $\nn$-graded minimal free resolution of the Stanley Reisner ring of a skeleton can be expressed as a $\mathbb{Z}$-linear combination of the corresponding Betti numbers of the original complex. This leads immediately to a new and direct proof that the property of being Cohen-Macaulay is inherited from the original complex.  
\item In Section \ref{SecMat} we see how our main result applies to truncations of matroids. We also explore whether a similar result can be obtained for matroid elongations.
\item In Section \ref{SecCounter} we give a counterexample demonstrating that our main result does not hold for the generalized skeleton ideals constructed in \cite{HVZ} and \cite{HJZ}.
\end{itemize} 
 
\section{Preliminaries}\label{Secprelim}
\subsection{Simplicial complexes}
\begin{definition}A \emph{simplicial complex} $\dd$ on $E=\{1,\ldots,n\}$ is a collection of subsets of $E$ that is closed under inclusion.
\end{definition}
We refer to the elements of $\dd$ as the \emph{faces} of $\dd$. A \emph{facet} of $\dd$ is a face that is not properly contained in another face, while a \emph{nonface} is a subset of $E$ that is not a face.

\begin{definition}
If $X\subseteq E$, then $\dd_{|X}=\{\sigma\subseteq X:\sigma \in \dd\}$ is itself a simplicial complex. We refer to $\dd_{|X}$ as the \emph{restriction of $\dd$ to $X$}.  
\end{definition}

\begin{definition}
Let $m$ be the cardinality of the largest face contained in $X\subseteq E$. The \emph{dimension} of $X$ is $\dim(X)=m-1$.  
\end{definition}
In particular, the dimension of a face $\sigma$ is equal to $|\sigma|-1$. We define $\dim(\dd)=\dim(E),$ and refer to this as the dimension of $\dd$. 
\begin{definition}[The $i$-skeleton of $\dd$]For $0\leq i\leq \dim(\dd)$, let the $i$-skeleton $\iskel{\dd}{i}$ be the simplicial complex \[\iskel{\dd}{i}=\{\sigma\in\dd:\dim(\sigma)\leq i\}.\]
\end{definition}
In particular, we have $\iskel{\dd}{d}=\dd$. The $1$-skeleton $\iskel{\dd}{1}$ is often referred to as the underlying graph of $\dd$.

\begin{remark}\label{remarkSigm}Whenever $\sigma \in \nnn$ the expression $|\sigma|$ shall signify the sum of the coordinates of $\sigma$. When, on the other hand, $\sigma\subseteq\{1\ldots n\}$, the expression $|\sigma|$ denotes the cardinality of $\sigma$. 
\end{remark}
 
\subsection{Matroids}There are numerous equivalent ways of defining a matroid. It is most convenient here to give the definition in terms of independent sets. For an introduction to matroid theory in general, we recommend e.g.~\cite{Oxl}.
\begin{definition}A \emph{matroid} $M$ consists of a finite set $E$ and a non-empty set $I(M)$ of subsets of $E$ such that:
\begin{itemize} \item $I(M)$ is a simplicial complex. \item If $I_{1},I_{2}\in \ind{M}$ and $|I_{1}|>|I_{2}|$, then there is an $x\in I_{1}\smallsetminus I_{2}$ such that $I_{2}\cup x\in \ind{M}$.\end{itemize} 
\end{definition}
The elements of $\ind{M}$ are referred to as the \emph{independent sets} (of $M$). The \emph{bases} of $M$ are the independent sets that are not contained in any other independent set: in other words, the facets of $\ind{M}$. Conversely, given the bases of a matroid, we find the independent sets to be those sets that are contained in a basis. We denote the bases of $M$ by $\bas{M}$. It is a fundamental result that all bases of a matroid have the same cardinality, which implies that $I(M)$ is a \emph{pure} simplicial complex.

The dual matroid $\dual{M}$ is the matroid on $E$ whose bases are the complements of the bases of $M$. Thus \[\bas{\dual{M}}=\{E\smallsetminus B:B\in \bas{M}\}.\] 
\begin{definition}For $X\subseteq E$, the rank function $r_{M}$ of $M$ is defined by \[r_{M}(X)=\max\{|I|:I \in \ind{M}, I\subseteq X\}.\]\end{definition}

Whenever the matroid $M$ is clear from the context, we omit the subscript and write simply $r(X)$. The rank $r(M)$ of $M$ itself is defined as $r(M)=r_{M}(E)$. Whenever $I(M)$ is considered as a simplicial complex we thus have $r(X)=\dim(X)+1$ for all $X\subseteq E$, and $r(M)=\dim(\ind{M})+1$.

\begin{definition}
If $X\subseteq E$, then $\{I\subseteq X:I \in \ind{M}\}$ form the set of independent sets of a matroid $M_{|X}$ on $X$. We refer to $M_{|X}$ as the \emph{restriction of $M$ to $X$}.  
\end{definition}
In \cite{Lar} the $i$th generalized Hamming weight of a linear code is generalized to matroids as follows.
\begin{definition}
For $1\leq i \leq n-r(M)$, the $i$th higher weight of $M$ is \[d_i(M)=\min\{|X|:X\subseteq E\text{ and }|X|-r(X)=i\}.\] 
\end{definition}
We refer to $\{d_i(M)\}$ as the \emph{higher weights} of $M$.
 
\begin{definition}[Truncation]
The $i$th truncation $\tru{M}{i}$ of $M$ is the matroid on $E$ whose independent sets consist of the independent sets of $M$ that have rank less than or equal to $r(M)-i$. In other words \[\ind{\tru{M}{i}}=\{X \subseteq E:r(X)=|X|, r(X)\leq r(M)-i\}.\]
\end{definition}
Observe that $\tru{M}{i}=\iskel{\ind{M}}{r(M)-i-1}$, whenever $I(M)$ is considered as a simplicial complex. That is, the $i$th truncation corresponds to the $(d-i)$-skeleton. 

\begin{definition}[Elongation]
For $0\leq i\leq n-r(M)$, let $\elo{M}{i}$ be the matroid whose independent sets are $I(\elo{M}{i})=\{\sigma\in E: n(\sigma)\leq i\}$. 
\end{definition}

Since $r(\elo{M}{i})=r(M)+i$, the matroid $\elo{M}{i}$ is commonly referred to as the \emph{elongation} of $M$ to rank $r(M)+i$.
It is straightforward to verify that for $i \in [0,\ldots,n-r(M)]$ we have $\elo{\overline{M}}{i}=\overline{\tru{M}{i}}$.

\subsection{The Stanley-Reisner ideal, Betti numbers, and the reduced chain complex}
Let $\dd$ be an abstract simplicial complex on $E=\{1,\ldots,n\}$. Let $\field{k}$ be a field, and let $S=\polr{k}{x}{n}$. By employing the standard abbreviated notation \[x_1^{\feit{a}(1)}x_2^{\feit{a}(2)}\cdots x_n^{\feit{a}(n)}=\mon{a}\] for monomials, we establish a $1-1$ connection between monomials of $S$ and vectors in $\nnn$. Furthermore, identifying a subset of $E$ with its indicator vector in $\nnn$ (as is done in Definition \ref{sridDef} below)  thus provides a $1-1$ connection between squarefree monomials of $S$ and subsets of $E$. 
\begin{definition}\label{sridDef}
Let $I_{\dd}$ be the ideal in $S$ generated by monomials corresponding to nonfaces of $\dd$. That is, let \[I_{\dd}=\langle\mon{\sigma}:\sigma\notin \dd\rangle.\] We refer to $I_{\dd}$ and $\srring{\dd}$, respectively, as the \emph{Stanley-Reisner ideal} and \emph{Stanley-Reisner ring} of $\dd$. 
\end{definition}

Being a (squarefree) monomial ideal, the Stanley-Reisner ideal, and thus also the Stanley-Reisner ring, permits both the standard $\nn$-grading and the standard $\nnn$-grading. For $\feit{b}\in\nnn$ let $S_{\feit{b}}$ be the $1$-dimensional $\field{k}$-vector space generated by $\feit{x}^{\feit{b}}$, and let $S(\feit{a})$, $S$ shifted by $\feit{a}$, be defined by $S(a)_{\feit{b}}=S_{\feit{a}+\feit{b}}$. Analogously, for $j\in\nn$ let $S_{i}$ be the $\field{k}$-vector space generated by monomials of degree $i$, and let $S(j)$ be defined by $S(j)_{i}=S_{i+j}$. For the remainder of this section let $N$ be an $\nnn$-graded $S$-module.
\begin{definition}\label{minfree}An \emph{($\nnn$- or $\nn$-)graded minimal free resolution} of $N$ is a left complex
\[\begin{CD}
0@<<<F_{0}@<\phi_1<<F_{1}@<\phi_2<<F_{2}@<<<\cdots@<\phi_l<<F_{l}@<<<0 
\end{CD}\]
with the following properties: \begin{itemize}
\item $F_i=\begin{cases}\bigoplus_{\mathbf{a}\in\nnn}S(-\mathbf{a})^{\Bb{i}{\feit{a}}}, \nnn\text{-graded resolution}\\
\bigoplus_{j\in\nn}S(-j)^{\Bb{i}{j}}, \nn\text{-graded resolution}\\	\end{cases}$
\item $\im\phi_i=\ker\phi_{i-1}$ for all $i\geq2$, and $F_0/\im\phi_1\cong N$ (Exact) 
\item $\im\phi_{i}\subseteq\feit{m}F_{i-1}$ (Minimal) 
\item \begin{flalign*}\phi_{i}\big((F_i)_{\feit{a}}\big)\subseteq&(F_{i-1})_{\feit{a}}\text{ (Degree preserving, $\nnn$-graded case)}\\
 \phi_{i}\big((F_i)_{j}\big)\subseteq&(F_{i-1})_{j} \text{ (Degree preserving, $\nn$-graded case)}.\end{flalign*}\end{itemize}
\end{definition} 

It follow from \cite[Theorem A.2.2]{HH} that the Betti numbers associated to a ($\nn$- or $\nnn$-graded) minimal free resolution are unique, in that any other minimal free resolution must have the same Betti numbers. We may therefore without ambiguity refer to $\{\BbModulk{i}{\feit{a}}{N}\}$ and $\{\BbModulk{i}{j}{N}\}$, respectively, as  the $\nnn$-graded and $\nn$-graded Betti numbers of $N$ (over $\field{k}$). Observe that \[\BbModulk{i}{j}{N}=\sum_{|\feit{a}|=j}\BbModulk{i}{\feit{a}}{N}\] where $|\feit{a}|=\feit{a}(1)+\feit{a}(2)+\cdots+\feit{a}(n)$ (see Remark \ref{remarkSigm}, above). Note also that for an $\nnn$-graded (that is, monomial) ideal $I\subseteq S$, we have $\BbModulk{i}{\sigma}{S/I}=\BbModulk{i-1}{\sigma}{I}$ for all $i\geq 1$, and $\BbModulk{0}{\sigma}{S/I}=\begin{cases}1, \sigma=\emptyset\\ 0, \sigma\neq\emptyset\end{cases}$. 

The $\nn$-graded Betti numbers of $N$ may be compactly presented in a so-called \emph{Betti table}: 
\[\beta[N;\field{k}]=\begin{array}{r|cccc}
& 0 & 1& \cdots & l\\
\hline
j & \BbModulk{0}{j}{N} & \BbModulk{1}{j+1}{N}&\cdots & \BbModulk{l}{j+l}{N}\\
j+1 & \BbModulk{0}{j+1}{N} & \BbModulk{1}{j+2}{N}&\cdots & \BbModulk{l}{j+l+1}{N}\\
\vdots& \vdots & \vdots & \cdots & \vdots\\
k & \BbModulk{0}{k}{N} & \BbModulk{1}{k+1}{N}&\cdots & \BbModulk{l}{k+l}{N}\\
\end{array}\]

By the (graded) \emph{Hilbert Syzygy Theorem} we have $F_i=0$ for all $i\geq n$. If $F_l\neq 0$ but $F_i=0$ for all $i>l$, we refer to $l$ as the \emph{length} of the minimal free resolution. It can be seen from e.g.~\cite[Corollary 1.8]{Eis} that the length of a minimal free resolution of $N$ equals its projective dimension ($\ppd N$). 

A sequence $f_1,\ldots,f_r\in \langle x_1,x_2,\ldots,x_n\rangle$ is said to be a \emph{regular $N$-sequence} if $f_{i+1}$ is not a zero-divisor on $N/(f_1N+\cdots+f_iN)$. 

\begin{definition}
The \emph{depth} of $N$ is the common length of a longest regular $N$-sequence. Whenever $N$ is $\nn$-graded the polynomials may be assumed to be homogeneous.
\end{definition}
In general we have $\depth N\leq \Kdim N$, where $\Kdim N$ denotes the Krull dimension of $N$. The following is a particular case of the famous \emph{Auslander-Buchsbaum Theorem}.
\begin{theorem}[Auslander-Buchsbaum]
\[\ppd N+\depth N=n.\] 
\end{theorem}
\begin{proof}
See e.g.~\cite[Corollary A.4.3]{HH}. 
\end{proof}
Note that the Krull dimension $\Kdim\srring{\dd}$ of $\srring{\dd}$ is one more than the dimension of $\dd$ (see \cite[Corollary 6.2.2]{HH}). The simplicial complex $\dd$ is said to be \emph{Cohen-Macaulay} if $\depth\srring{\dd} =\Kdim \srring{\dd}$. That is, if $\srring{\dd}$ is Cohen-Macaulay as an $S$-module.  
   
\begin{definition}
Let $\ino{\dd}{i}$ denote the set of $i$-dimensional faces of $\dd$. That is, \[\ino{\dd}{i}=\{\sigma \in \dd:|\sigma|=i+1\}.\] Let $\field{k}^{\ino{\dd}{i}}$ be the free $\field{k}$-vector space on $\ino{\dd}{i}$. The \emph{(reduced) chain complex} of $M$ over $\field{k}$ is the complex \[\minCDarrowwidth13pt\begin{CD}0
@<<<\field{k}^{\ino{\dd}{-1}}@<\delta_{0}<<\cdots@<<<\field{k}^{\ino{\dd}{i-1}}@<\delta_{i}<<\field{k}^{\ino{\dd}{i}}@<<<\cdots@<\delta_{\dim(\dd)}<<\field{k}^{\ino{\dd}{\dim(\dd)}}@<<<0\end{CD},\] where the boundary maps $\delta_{i}$ are defined as follows: With the natural ordering on $E$, set $\sign(j,\sigma)=(-1)^{r-1}$ if $j$ is the $r$th element of $\sigma\subseteq E$, and let \[\delta_{i}(\sigma)=\sum_{j\in \sigma}\sign(j,\sigma)\;\sigma\smallsetminus j.\] Extending $\delta_{i}$ $\field{k}$-linearly, we obtain a $\field{k}$-linear map from $\field{k}^{\ino{\dd}{i}}$ to $\field{k}^{\ino{\dd}{i-1}}$.
\end{definition}

\begin{definition}
The $i$th \emph{reduced homology} of $\dd$ over $\field{k}$ is the vector space \[\tilde{H}_{i}(\dd;\field{k})=\ker(\delta_{i})/\im(\delta_{i+1}).\]
\end{definition}
The following is one of the most celebrated results in the intersection between algebra and combinatorics.
\begin{theorem}[Hochster's formula]\label{Hochster}\[\Bringk{i}{\sigma}{\dd}=\Bnumk{i-1}{\sigma}{\dd}=\dim_{\field{k}}\tilde{H}_{|\sigma|-i-1}(\dd_{|\sigma};\field{k}).\]
\end{theorem}
\begin{proof}
See \cite[Corollary 5.12]{MS} and \cite[p.~81]{HH}.
\end{proof}

\section{Betti numbers of $i$-skeletons}\label{SecSimp}
\noindent Let $\dd$ be a $d$-dimensional simplicial complex on $\{1,\ldots,n\}$, and let $\field{k}$ be a field. In this section we shall demonstrate how each of the Betti numbers of $\srring{\iskel{\dd}{d-1}}$ can be expressed as a $\mathbb{Z}$-linear combination of the Betti numbers of $\srring{\dd}$. 

\subsection{The first rows of the Betti table}
\begin{lemma}\label{ansikter}\[\tilde{H}_{i}(\dd_{|\sigma};\field{k})=\tilde{H}_{i}(\trures{\dd}{d-1}{\sigma};\field{k})\] for all $0\leq i\leq d-2$.
\end{lemma}
\begin{proof}
By the definition of a skeleton we have $\ino{\dd_{|\sigma}}{i}=\ino{\trures{\dd}{d-1}{\sigma}}{i}$ and thus also $\field{k}^{\ino{\dd_{|\sigma}}{i}}=\field{k}^{\ino{\trures{\dd}{d-1}{\sigma}}{i}}$, for all $-1\leq i\leq d-1$. In other words, the reduced chain complexes of $\dd_{|\sigma}$ and ${\iskel{\dd}{d-1}}_{|\sigma}$ are identical except for in homological degree $d$. The result follows.
\end{proof}

\begin{proposition}\label{likebortsettfrasiste}
For all $i$ and $j \leq d+i-1$ we have \[\Bringk{i}{j}{\dd}=\Bringk{i}{j}{\iskel{\dd}{d-1}}.\] 
\end{proposition}
\begin{proof}
If $j \leq d+i-1$ then $j-i-1\leq d-2$. By Theorem \ref{Hochster} and Lemma \ref{ansikter} then, we have \begin{flalign*}\Bringk{i}{j}{\dd}&=\sum_{|\sigma|=j}\Bringk{i}{\sigma}{\dd}\\
&=\sum_{|\sigma|=j}\dim_{\field{k}}\tilde{H}_{|\sigma|-i-1}(\dd_{|\sigma};\field{k})\\
&=\sum_{|\sigma|=j}\dim_{\field{k}}\tilde{H}_{|\sigma|-i-1}(\trures{\dd}{d-1}{\sigma};\field{k})\\
&=\sum_{|\sigma|=j}\Bringk{i}{\sigma}{\tru{\dd}{d-1}}\\&=\Bringk{i}{j}{\iskel{\dd}{d-1}}.\end{flalign*}
\end{proof}

\subsection{The final row of the Betti table}
The Hilbert series of $S/\srid{\dd}$ over $\field{k}$ is $\hilb{S/\srid{\dd}}=\sum_{i\in\mathbb{Z}}\dim_{\field{k}}(S/\srid{\dd})_{i}\;t^{i}$. Let $\ff{\dd}{i}=|\ino{\dd}{i}|$. By \cite[Section 6.1.3, Equation (6.3)]{HH} we have \[\hilb{S/\srid{\dd}}=\frac{\sum_{i=0}^{n}(-1)^i\sum_{j}\BbModulk{i}{j}{S/\srid{\dd}}}{(1-t)^n}.\]
On the other hand, we see from \cite[Proposition 6.2.1]{HH} that \[\hilb{S/\srid{\dd}}=\frac{\sum_{i=0}^{d+1}\ff{\dd}{i-1}t^i(1-t)^{d+1-i}}{(1-t)^{d+1}}.\] Combined, these two equations imply
\begin{equation}\label{A}
\sum_{i=0}^{d+1}\ff{\dd}{i-1}t^{i}(1-t)^{n-i}=\sum_{i=0}^{n}(-1)^{i}\sum_{j}\Bringk{i}{j}{\dd}t^{j}, 
\end{equation}
and
\begin{equation}\label{B}
\sum_{i=0}^{d}\ff{\tru{\dd}{d-1}}{i-1}t^{i}(1-t)^{n-i}=\sum_{i=0}^{n}(-1)^{i}\sum_{j}\Bringk{i}{j}{\tru{\dd}{d-1}}t^{j}. 
\end{equation}

\begin{remark}From here on we shall employ the convention that $i!=0$ for $i<0$, and that $\binom{j}{k}=0$ if one or both of $j$ and $k$ is negative.\end{remark} 

Differentiating both sides of equation (\ref{A}) $n-d-1$ times, we get \begin{flalign*}&\sum_{i=0}^{d+1}\ff{\dd}{i-1}\sum_{l=0}^{n-d-1}(-1)^{l}\binom{n-d-1}{l}\frac{i!(n-i)!}{(i-n+d+1+l)!(n-i-l)!}t^{i-n+d+1+l}(1-t)^{n-i-l}\\&=\sum_{i=0}^{n}(-1)^{i}\sum_{j}\Bringk{i}{j}{\dd}\frac{j!}{(j-(n-d-1))!}t^{j-n+d+1}. \end{flalign*} When evaluated at $t=1$, the left side of the above equation is $0$ except when $i=d+1$ and $l=n-d-1$. Thus, we have \[(-1)^{n-d-1}(n-d-1)!\ff{\dd}{d}=\sum_{i=0}^{n}(-1)^{i}\sum_{j\geq n-d-1}\Bringk{i}{j}{\dd}\frac{j!}{(j-(n-d-1))!},\] and \[\ff{\dd}{d}=\sum_{i=0}^{n}(-1)^{n+d+i+1}\sum_{j\geq n-d-1}\binom{j}{n-d-1}\Bringk{i}{j}{\dd}.\] 
\begin{lemma}\label{nedenfor}
For all $i$ and $j\geq d+i+2$ we have \[\Bringk{i}{j}{\dd}=0.\] 
\end{lemma}
\begin{proof}
If $|\sigma|\geq d+i+2$, then $|\sigma|-i-1\geq \dim(\dd)+1$, which implies \[\dim_{\field{k}}\tilde{H}_{|\sigma|-i-1}(\dd_{|\sigma};\field{k})=0.\] So by Hochster's formula we have that if $j\geq d+i+2$ then 
\[\Bringk{i}{j}{\dd}=\sum_{|\sigma|=j}\Bringk{i}{\sigma}{\dd}=\sum_{|\sigma|=j}\dim_{\field{k}}\tilde{H}_{|\sigma|-i-1}(\dd_{|\sigma};\field{k})=0.\]
\end{proof}
According to Proposition \ref{likebortsettfrasiste} and Lemma \ref{nedenfor}, and because $\ff{\dd}{i}=\ff{\iskel{\dd}{d-1}}{i}$  for all $i\neq d$, subtracting equation (\ref{B}) from equation (\ref{A}) yields \begin{flalign*}\label{pent}\ff{\dd}{d}t^{d+1}(1-t)^{n-d-1}=&\sum_{i=0}^{n}(-1)^{i}\big(\Bringk{i}{d+i}{\dd}-\Bringk{i}{d+i}{\iskel{\dd}{d-1}}\big)t^{d+i}\\&+\sum_{i=0}^{n}(-1)^{i}\Bringk{i}{d+i+1}{\dd}t^{d+i+1}.\end{flalign*} Let $1\leq u\leq n$. Differentiating both sides of the above equation $d+u$ times yields \begin{flalign*}&\ff{\dd}{d}\sum_{l=0}^{d+u}(-1)^{l}\binom{d+u}{l}\frac{(d+1)!(n-d-1)!}{(l-u+1)!(n-d-1-l)!}t^{l-u+1}(1-t)^{n-d-1-l}\\=&\sum_{i=u}^{n}(-1)^{i}\big(\Bringk{i}{d+i}{\dd}-\Bringk{i}{d+i+1}{\iskel{\dd}{d-1}}\big)\frac{(d+i)!}{(i-u)!} t^{i-u}\\&+\sum_{i=u-1}^{n}(-1)^{i}\Bringk{i}{d+i+1}{\dd}\frac{(d+i+1)!}{(i-u+1)!}t^{i-u+1}.\end{flalign*} Evaluating at $t=0$, we get \begin{flalign*}&\delta^{\prime}*\left((-1)^{u-1}\ff{\dd}{d}\frac{(d+u)!(n-d-1)!}{(u-1)!(n-d-u)!}\right)\\=&(-1)^{u}\big(\Bringk{u}{
d+u}{\dd}-\Bringk{u}{d+u}{\iskel{\dd}{d-1}}\big)(d+u)!\\&+(-1)^{u-1}\Bringk{u-1}{d+u}{\dd}(d+u)!,\end{flalign*} where \[\delta^{\prime}=\begin{cases}
1,& 1\leq u\leq n-d\\
0,& u> n-d
\end{cases}
.\]

Summarizing the above:
\begin{proposition}\label{main}
For $1\leq u\leq n$, we have \[\Bringk{u}{d+u}{\tru{\dd}{d-1}}=\Bringk{u}{d+u}{\dd}-\Bringk{u-1}{d+u}{\dd}+\binom{n-d-1}{u-1}\delta,\] where \[\delta=\begin{cases}\ff{\dd}{d}=\sum_{i=0}^{n}(-1)^{n+d+i+1}\sum_{j\geq n-d-1}\binom{j}{n-d-1}\Bringk{i}{j}{\dd},& 1\leq u\leq n-d\\0,& u> n-d.\end{cases}\]
\end{proposition}
Bringing together Propositions \ref{likebortsettfrasiste} and \ref{main}, we get
\begin{theorem}\label{mainThe}
For all $i\geq1$, we have
\[\Bringk{i}{j}{\tru{\dd}{d-1}}=\begin{cases}
\Bringk{i}{j}{\dd}, & \text{$j \leq d+i-1$} \\ 
\Bringk{i}{d+i}{\dd}-\Bringk{i-1}{d+i}{\dd}+\binom{n-d-1}{i-1}\delta, & \text{$j=d+i$,} \\
0, & \text{$j \geq d+i-1$}
\end{cases}\] where \[\delta=\begin{cases}\ff{\dd}{d}=\sum_{k=0}^{n}(-1)^{n+d+k+1}\sum_{l\geq n-d-1}\binom{l}{n-d-1}\Bringk{k}{l}{\dd},& 1\leq i\leq n-d\\0,& i> n-d.\end{cases}\]
\end{theorem}

\begin{example}\label{exTriang}
Let $T$ be one of the two irreducible triangulations of the real projective plane (see \cite{Bar}) -- namely the one corresponding to an embedding of the complete graph on $6$ vertices. Clearly then, we have $n=6$ and $d=2$. The Betti table of $\srring{T}$ over $\mathbb{F}_{3}$ is \[\beta[\srring{T};\mathbb{F}_3]=\begin{array}{r|cccc}
 & 0 & 1 & 2& 3 \\
\hline
0 & 1 & 0 & 0 & 0\\
1 & 0 & 0 & 0 & 0\\
2 & 0 & 0 & 0 & 0\\
3 & 0 & 10 & 15 & 6\\
\end{array}.\]
In this case \[\ff{\dd}{d}=\binom{4}{3}\BringKropp{1}{4}{T}{\mathbb{F}_3}-\binom{5}{3}\BringKropp{2}{5}{T}{\mathbb{F}_3}+\binom{6}{3}\BringKropp{3}{6}{T}{\mathbb{F}_3}=10.\] By Theorem \ref{mainThe}, the Betti numbers of $\srring{\iskel{T}{1}}$ are \begin{flalign*}
\BringKropp{1}{4}{\iskel{T}{1}}{\mathbb{F}_3}=&\BringKropp{1}{4}{T}{\mathbb{F}_3}+\binom{3}{0}\delta=10+10.\\
\BringKropp{2}{5}{\iskel{T}{1}}{\mathbb{F}_3}=&\BringKropp{2}{5}{T}{\mathbb{F}_3}-\BringKropp{1}{5}{T}{\mathbb{F}_3}+\binom{3}{1}\delta=15+30.\\                                                                                                                                                                                                  
\BringKropp{3}{6}{\iskel{T}{1}}{\mathbb{F}_3}=&\BringKropp{3}{6}{T}{\mathbb{F}_3}-\BringKropp{2}{6}{T}{\mathbb{F}_3}+\binom{3}{2}\delta=6-0+30.\\ 
\BringKropp{4}{7}{\iskel{T}{1}}{\mathbb{F}_3}=&\BringKropp{4}{7}{T}{\mathbb{F}_3}-\BringKropp{3}{7}{T}{\mathbb{F}_3}+\binom{3}{3}\delta=0-0+10.\\
\end{flalign*}
\[\beta[\srring{\iskel{T}{1}};\mathbb{F}_3]=\begin{array}{r|ccccc}
& 0 & 1 & 2 & 3 & 4 \\
\hline
0 & 1 & 0 & 0 & 0 & 0 \\
1 & 0 & 0 & 0 & 0 & 0 \\
2 & 0 & 0 & 0 & 0 & 0 \\
3 & 0 & 20 & 45 & 36 & 10\\
\end{array}.\]
\end{example}
\begin{remark}
Observe that as \[\beta[\srring{T};\mathbb{F}_2]=\begin{array}{r|ccccc}
& 0 & 1 & 2 & 3 & 4\\
\hline
0 & 1 & 0 & 0 & 0 & 0\\
1 & 0 & 0 & 0 & 0 & 0\\
2 & 0 & 0 & 0 & 0 & 0\\
3 & 0 & 10 & 15 & 6 & 1\\
4 & 0 & 0 & 0 & 1 & 0\\
\end{array},\] the simplicial complex $T$ of Example \ref{exTriang} is an example of a pure simplicial complex whose Betti numbers depend upon the field $\field{k}$ -- as opposed to what is the case for matroids. 
\end{remark}

\subsection{The projective dimension of skeletons}
Let $\ppd\srring{\dd}$ denote the projective dimension of $\srring{\dd}$. By Auslander-Buchsbaum Theorem we have \begin{flalign*}\ppd{\srring{\dd}}=&n-\depth \srring{\dd}\\
\geq&n-\Kdim \srring{\dd}\\=& n-(d+1), \end{flalign*} so $n-d-1\leq \ppd{\srring{\dd}}\leq n$. 

As for the skeletons, we have
\begin{corollary}\label{pdleq}
\[\ppd{\srring{\iskel{\dd}{d-1}}}\leq 1+\ppd{\srring{\dd}}.\]
\end{corollary}
\begin{proof}
Let $p=\ppd{\srring{\dd}}$. By Proposition \ref{likebortsettfrasiste} it suffices to show that \[\Bringk{p+2}{d+p+2}{\iskel{\dd}{d-1}}=0.\] But by Theorem \ref{main}, we have \begin{flalign*}\Bringk{p+2}{d+p+2}{\iskel{\dd}{d-1}}=&\Bringk{p+2}{d+p+2}{\dd}-\Bringk{p+1}{d+p+2}{\dd}+\delta\\=&0-0-\delta=0,\end{flalign*} 
where the last equality is due to $p+2> n-d$. 
\end{proof}

\begin{corollary}\label{CM}
If $\dd$ is Cohen-Macaulay, then so is $\iskel{\dd}{d-1}$. 
\end{corollary}
\begin{proof}
Let $\dd$ be a simplicial complex with $\dim(\dd)=d$ and $\depth\srring{\dd}=\Kdim\srring{\dd}$. As $\Kdim\srring{\iskel{\dd}{d-1}}=d$, we only need to prove that $\depth\srring{\iskel{\dd}{d-1}}=d$ as well.  

Since $\depth \srring{\iskel{\dd}{d-1}}\leq \Kdim \srring{\iskel{\dd}{d-1}}=d,$ we have by the Auslander-Buchsbaum Theorem that $\ppd \srring{\iskel{\dd}{d-1}}\geq n-d.$ On the other hand, since \begin{flalign*}\ppd \srring{\dd}=&n-\depth \srring{\dd}\\=&n-\Kdim\srring{\dd}\\=&n-(d+1),\end{flalign*} we see from Corollary \ref{pdleq} that $\ppd\srring{\iskel{\dd}{d-1}}\leq n-d$. We conclude that \[\ppd\srring{\iskel{\dd}{d-1}}= n-d\] and, by Auslander-Buchsbaum again, that $\depth \srring{\iskel{\dd}{d-1}}=d.$ 
\end{proof}
  
\section{Betti numbers of truncations and elongations of matroids}\label{SecMat}
\noindent Let $M$ be a matroid on $\{1,\ldots,n\}$, with $r(M)=k$. As was established in \cite{Bjoe}, the dimension of $\tilde{H}_{i}(M;\field{k})$ is in fact independent of the field $\field{k}$ . Thus \emph{for matroids, the ($\nn$- or $\nnn$-graded) Betti numbers are not only unique, but independent of the choice of field}. We shall therefore omit referring to or specifying a particular field $\field{k}$ throughout this section. By a slight abuse of notation we shall denote the Stanley-Reisner ideal associated to the set of independent sets $I(M)$ of $M$ simply by $\srid{M}$.
\subsection{Truncations}
Note that the $i$th truncation of $M$ corresponds to the $(k-i-1)$-skeleton of $I(M)$, a fact which enables us to invoke Theorem \ref{mainThe}. In addition, it follows from \cite[Corollary 3(b)]{JV} that the minimal free resolutions of $\srring{M}$ have length $n-k$. We thus have 
\begin{proposition}\label{hovedForMatroider}
For all $i$, we have
\[\Bring{i}{j}{\tru{M}{1}}=\begin{cases}
\Bring{i}{j}{M}, & \text{$j \leq k+i-2$.} \\ 
\Bring{i}{k+i-1}{M}-\Bring{i-1}{k+i-1}{M}\\+\binom{n-k}{i-1}\Big(\sum_{u=0}^{n-k}(-1)^{n+k+u}\sum_{v\geq n-k}\binom{v}{n-k}\Bring{u}{v}{M}\Big), & \text{$j=k+i-1$.} \\
0, & \text{$j\geq k+i$.}
\end{cases}
\]
\end{proposition}
\hide{In \cite{JRV} we demonstrate that the generalized Hamming weights of the first elongation $\elo{M}{1}$ of $M$ is determined by those of $M$ very directly by $d_i(\elo{M}{1})=d_{i+1}(M)$ for all $1\leq i\leq n-k-1$. 

From Proposition \ref{hovedForMatroider}, we get a similar result for truncations.}
\begin{corollary}
For all $1\leq i\leq n-k+1$, we have \[d_i(\tru{M}{1})=\min\{d_i(M),k+i-1\}.\] 
\end{corollary}
\begin{proof}
By \cite[Theorem 4]{JV} we have $d_i(\tru{M}{1})=\min\{j:\Bring{i}{j}{\tru{M}{1}}\neq0\}.$ The result now follows immediately from Proposition \ref{hovedForMatroider}.
\end{proof}

\subsection{Elongations}\label{SecElong}
When it comes to elongations, the Betti numbers of $M$ provide far less information about the Betti numbers of $\elo{M}{1}$ than what was the case with truncations. We do however have the following.
\begin{proposition}\label{bettitables}
For $i\geq1$, \[\Bnum{i}{j}{\elo{M}{l}}\neq0 \iff \Bnum{i-1}{j}{\elo{M}{l+1}}\neq0.\] 
\end{proposition}
\begin{proof}
According to \cite[Theorem 1]{JV}, we have that \[\Bnum{i}{\sigma}{M}\neq 0 \iff \sigma \text{ is minimal with the property that } n_{M}(\sigma)=i+1.\] Since $\beta_{i,j}=\sum_{|\sigma|=j}\beta_{i,\sigma},$ we see that 
\begin{flalign*}
\Bnum{i}{j}{\elo{M}{l}}&\neq 0\\ &\iff\\ 
\text{There is a } \sigma \text{ such that } |\sigma|=j \text{ and }& \sigma \text{ is minimal with the property that } n_{\elo{M}{l}}(\sigma)=i+1\\ &\iff\\
\text{There is a } \sigma \text{ such that } |\sigma|=j \text{ and }& \sigma \text{ is minimal with the property that } n_{\elo{M}{l+1}}(\sigma)=i\\ &\iff\\
\Bnum{i-1}{j}{\elo{M}{l+1}}&\neq 0.
\end{flalign*}
\end{proof}

In terms of Betti tables, this implies that when it comes to zeros and nonzeros the Betti table of $\srid{\elo{M}{i+1}}$ is equal to the table you get by deleting the first column from the table of $\srid{M_i}$. As the following counterexample (computed using MAGMA \cite{MAGMA}) demonstrates, there can be no result for elongations analogous to Theorem \ref{mainThe}.

Let $M$ and $N$ be the matroids on $\{1,\ldots, 8\}$ with bases 
\begin{flalign*}B(M)=\big\{
        &\{ 1, 3, 4, 6, 7 \},
        \{ 1, 2, 3, 6, 8 \},
        \{ 1, 2, 3, 4, 8 \},
        \{ 1, 2, 3, 5, 8 \},
        \{ 1, 2, 5, 6, 8 \},        \\
        &\{ 1, 2, 3, 4, 7 \},
        \{ 1, 2, 3, 5, 7 \},
        \{ 1, 2, 5, 6, 7 \},
        \{ 1, 3, 4, 5, 7 \},
        \{ 1, 3, 4, 6, 8 \},
        \\
        &\{ 1, 2, 4, 6, 8 \},
        \{ 1, 2, 4, 6, 7 \},
        \{ 1, 3, 4, 5, 8 \},
        \{ 1, 2, 4, 5, 7 \},
        \{ 1, 4, 5, 6, 7 \},
        \\
        &\{ 1, 2, 3, 6, 7 \},
        \{ 1, 3, 5, 6, 7 \},
        \{ 1, 4, 5, 6, 8 \},
        \{ 1, 3, 5, 6, 8 \},
        \{ 1, 2, 4, 5, 8 \}\big\}
\end{flalign*} and 
\begin{flalign*}B(N)=\big\{
&\{ 1, 3, 4, 6, 7 \},
        \{ 1, 2, 3, 4, 8 \},
        \{ 1, 2, 3, 5, 8 \},
        \{ 1, 2, 5, 6, 8 \},
        \{ 1, 2, 3, 4, 7 \},
        \\
        &\{ 1, 2, 3, 5, 7 \},
        \{ 1, 2, 5, 6, 7 \},
        \{ 1, 3, 4, 5, 7 \},
        \{ 1, 3, 4, 6, 8 \},
        \{ 1, 2, 4, 6, 8 \},
       \\
        & \{ 1, 2, 4, 6, 7 \},
        \{ 1, 3, 4, 5, 8 \},
        \{ 1, 2, 4, 5, 7 \},
        \{ 1, 3, 4, 5, 6 \},
        \{ 1, 2, 4, 5, 6 \},
        \\
        &\{ 1, 3, 5, 6, 7 \},
        \{ 1, 2, 3, 5, 6 \},
        \{ 1, 2, 3, 4, 6 \},
        \{ 1, 3, 5, 6, 8 \},
        \{ 1, 2, 4, 5, 8 \}
        \big\}.
\end{flalign*} Both $\srid{M}$ and $\srid{N}$ have Betti table 
\[\begin{array}{r|ccc}
& 0 & 1& 2 \\
\hline
2 & 1 & 0&0\\
3 & 0 & 0&0\\
4 & 1 & 4&0\\
5 & 0 &5&4\\
\end{array},\]
but while $\srid{\elo{M}{1}}$ has Betti table 
\[\begin{array}{r|cc}
& 1& 2 \\
\hline
5  & 1&0\\
6  &5&5\\
\end{array}\] the ideal $\srid{\elo{N}{1}}$ has Betti table \[\begin{array}{r|cc}
& 1& 2 \\
\hline
5  & 2&0\\
6  &3&4\\
\end{array}.\] This shows that the Betti numbers associated to a matroid do not determine those associated to its elongation.

\section{The $i$th skeleton ideal}\label{SecCounter}
\noindent As mentioned in the introduction, the Stanley-Reisner ideal of a skeleton is generalized in \cite{HVZ} and \cite{HJZ} to arbitrary monomial ideals. We shall briefly describe the construction as it is found in the above papers, and present a counterexample showing that our main result does not extend to these ideals.

For $\feit{a}, \feit{b}\in\nnn$ we say that $\feit{a}\leq\feit{b}$ if $\feit{a}(i)\leq \feit{b}(i)$ for $1\leq i \leq n$. Clearly, this constitutes a partial order on $\nnn$. Let $I,J\subseteq S$ be monomial ideals with (unique) minimal generating sets $\{\feit{x}^{\feit{a}_1},\ldots,\feit{x}^{\feit{a}_r}\}$ and $\{\feit{x}^{\feit{b}_1},\ldots,\feit{x}^{\feit{b}_s}\}$, respectively, and let $\feit{g}\in\nnn$ be such that $\feit{a}_i\leq \feit{g}$ and $\feit{b}_j\leq \feit{g}$ for all $1\leq i\leq r$, $1\leq j\leq s$. Define the \emph{characteristic poset} $P^{\feit{g}}_{J/I}$ of $J/I$ with respect to $\feit{g}$ to be \[P^{\feit{g}}_{J/I}=\{\feit{b}\in\nnn:\feit{b}\leq \feit{g},\feit{b}\geq \feit{b}_j \text{ for some }j,\feit{b}\not\geq \feit{a}_i \text{ for all }i\}.\] For $\feit{b}\in\nnn$, let $\rho(\feit{b})=|i:\feit{b}(i)=\feit{g}(i)|$. It is demonstrated in \cite[Corollary 2.6]{HVZ} that $\Kdim J/I=\max\{\rho(\feit{b}):\feit{b}\in P^{\feit{g}}_{J/I}\}$.

The $j$th \emph{generalized skeleton ideal} $\srid{j}$ is the ideal generated by $\{\feit{x}^{\feit{a}_1},\ldots,\feit{x}^{\feit{a}_r}\}\cup\{\feit{x}^{\feit{b}}:\feit{b}\in\nnn, \rho(\feit{b})>j\}.$ By \cite[Corollary 2.5]{HJZ} these ideals form a chain $I=\srid{d}\subseteq \srid{d-1}\subseteq\cdots\subseteq \srid 0\subseteq S$ with the property that $\srring{j}$ is Cohen-Macauley for all $j\leq \depth\srring{}$, and $\depth\srring{}=\max\{j:\srring{j} \text{ is Cohen-Macauley}\}$. In other words, these ideals successfully generalize \eqref{Hibis} from the introduction. Furthermore, in the special case $J=S$, $I=\srid{\dd}$, and $\feit{g}=(1,1,\ldots,1)$, we have $\srid{j}=\srid{\iskel{\dd}{j}}$.

Now, let $M$, $N$ be the matroids on $\{1,\ldots,6\}$ with \begin{flalign*}B(M)=\big\{&\{ 1, 3, 6 \},\{ 1, 3, 5 \},\{ 4, 5, 6 \},\{ 1, 3, 4 \},\{ 2, 3, 6 \},\{ 1, 2, 5 \},\\
    &\{ 2, 4, 6 \},\{ 1, 4, 6 \},\{ 3, 5, 6 \},\{ 2, 3, 4 \},\{ 1, 2, 3 \},\{ 1, 5, 6 \},\\
    &\{ 3, 4, 5 \}, \{ 1, 4, 5 \}, \{ 1, 2, 4 \}, \{ 2, 5, 6 \},\{ 2, 3, 5 \},\{ 3, 4, 6 \} \big\}
\end{flalign*}  and \begin{flalign*}B(N)=\big\{&\{ 1, 3, 6 \}, \{ 1, 3, 5 \},\{ 4, 5, 6 \},\{ 1, 3, 4 \}, \{ 1, 2, 6 \},\{ 2, 3, 6 \},\\
    &\{ 1, 2, 5 \},\{ 2, 4, 6 \},\{ 3, 5, 6 \},\{ 2, 3, 4 \},\{ 1, 2, 3 \},\{ 1, 5, 6 \},\\
    &\{ 3, 4, 5 \},\{ 1, 4, 5 \},\{ 1, 2, 4 \},\{ 2, 5, 6 \},\{ 2, 4, 5 \},\{ 3, 4, 6 \} \big\}.
\end{flalign*}
Then \[\beta[\srring{M}]=\begin{array}{r|cccc}
& 0 & 1 & 2 & 3  \\
\hline
0 & 1 & 0 & 0 & 0  \\
1 & 0 & 0 & 0 & 0  \\
2 & 0 & 2 & 0 & 0  \\
3 & 0 & 9 & 18 & 8 \\
\end{array}=\beta[\srring{N}],\] for all base fields $\field{k}$. However, if we take $\feit{g}=(1,2,1,1,1,1)$ and $J=S=\polr{Q}{x}{n}$ in the above construction, we get 
\[\beta[S/(\srid{M})_{(1)};\mathbb{Q}]=\begin{array}{r|ccccc}
& 0 & 1 & 2 & 3 & 4 \\
\hline
0 & 1 & 0 & 0 & 0 & 0   \\
1 & 0 & 0 & 0 & 0 & 0 \\
2 & 0 & 12 & 30 & 12 & 2  \\
3 & 0 & 17 & 24 & 24 & 8\\
\end{array}\] while \[\beta[S/(\srid{N})_{(1)};\mathbb{Q}]=\begin{array}{r|ccccc}
& 0 & 1 & 2 & 3 & 4  \\
\hline
0 & 1 & 0 & 0 & 0 & 0   \\
1 & 0 & 0 & 0 & 0 & 0 \\
2 & 0 & 11 & 27 & 9 & 1  \\
3 & 0 & 18 & 27 & 27 & 9\\
\end{array}.\]
We conclude that the statement of Theorem \ref{mainThe} does not necessarily hold if one replaces the Stanley-Reisner ideals of skeletons with generalized skeleton ideals.

\end{document}